\documentclass[11pt,twoside]{amsart}

\usepackage[all]{xy}
\usepackage{amsmath}
\usepackage{amsfonts}
\usepackage{amssymb}
\usepackage{amsthm}
\usepackage{color}
\usepackage{enumitem}
\usepackage{hyperref}
\usepackage{fullpage}
\usepackage{graphicx}
\usepackage{url}

\usepackage{pgfplots}
\pgfplotsset{width=9cm,compat=1.9}
\pgfplotsset{holdot/.style={color=black,fill=white,only marks,mark=*}}

\setlist[enumerate,1]{label=(\roman*)}
\setlist[enumerate,2]{label=(\alph*)}
\setlist[enumerate,3]{label=(\Roman*)}
\setlist[enumerate,4]{label=(\Alph*)}

\theoremstyle{definition}
\newtheorem{defn}{Definition}[section]
\newtheorem{ex}[defn]{Example}
\newtheorem{rmk}[defn]{Remark}

\theoremstyle{plain}
\newtheorem{thm}[defn]{Theorem}

\newtheorem{lem}[defn]{Lemma}
\newtheorem{prop}[defn]{Proposition}

\def\C{\ensuremath{\mathbb{C}}}
\def\D{\ensuremath{\mathbb{D}}}

\def\P{\ensuremath{\mathbb{P}}}

\def\R{\ensuremath{\mathbb{R}}}
\def\Z{\ensuremath{\mathbb{Z}}}

\def\AA{\ensuremath{\mathcal A}}

\def\FF{\ensuremath{\mathcal F}}

\def\HH{\ensuremath{\mathcal H}}
\def\II{\ensuremath{\mathcal I}}

\def\OO{\ensuremath{\mathcal O}}
\def\PP{\ensuremath{\mathcal P}}

\def\TT{\ensuremath{\mathcal T}}

\def\Aut{\mathop{\mathrm{Aut}}\nolimits}

\def\c{\mathop{\mathrm{c}}\nolimits}
\def\ch{\mathop{\mathrm{ch}}\nolimits}

\def\Coh{\mathop{\mathrm{Coh}}\nolimits}

\def\Db{\mathrm{D}^{\mathrm{b}}}
\def\deg{\mathop{\mathrm{deg}}\nolimits}

\def\dim{\mathop{\mathrm{dim}}\nolimits}

\def\ext{\mathop{\mathrm{ext}}\nolimits}
\def\Ext{\mathop{\mathrm{Ext}}\nolimits}

\def\Gr{\mathop{\mathrm{Gr}}\nolimits}

\def\hom{\mathop{\mathrm{hom}}\nolimits}
\def\Hom{\mathop{\mathrm{Hom}}\nolimits}

\def\min{\mathop{\mathrm{min}}\nolimits}

\def\PGL{\mathop{\mathrm{PGL}}}

\def\Pic{\mathop{\mathrm{Pic}}\nolimits}

\def\rk{\mathop{\mathrm{rk}}}

\def\td{\mathop{\mathrm{td}}\nolimits}

\def\Stab{\mathop{\mathrm{Stab}}\nolimits}

\def\into{\ensuremath{\hookrightarrow}}
\def\onto{\ensuremath{\twoheadrightarrow}}

\title{Stability and Applications}

\author{Emanuele Macr\`i}
\address{Universit\'e Paris-Saclay, CNRS, Laboratoire de math\'ematiques d'Orsay, 91405 Orsay, France}
\email{emanuele.macri@universite-paris-saclay.fr}
\urladdr{https://www.math.u-psud.fr/~macri/}

\author{Benjamin Schmidt}
\address{Gottfried Wilhelm Leibniz Universit\"at Hannover, Institut f\"ur Algebraische Geometrie, Welfengarten 1, 30167 Hannover, Germany}
\email{bschmidt@math.uni-hannover.de}
\urladdr{https://sites.google.com/site/benjaminschmidtmath/}


\thanks{E.M.~was partially supported by the NSF FRG-grant DMS-1664215 and, during the writing of this paper, by the Institut des Hautes \'Etudes Scientifiques (IH\'ES) and by a Poste Rouge CNRS at Universit\'e Paris-Sud.
B.S.~was partially supported by an AMS-Simons Travel Grant.}


\begin{document}

\begin{abstract}
We give a brief overview of Bridgeland's theory of stability conditions, focusing on applications to algebraic geometry. We sketch the basic ideas in Bayer's proof of the Brill--Noether Theorem and in the authors' proof of a theorem by Gruson--Peskine and Harris on the genus of space curves.

This note originated from the lecture of the first author at the conference \emph{From Algebraic Geometry to Vision and AI: A Symposium Celebrating the Mathematical Work of David Mumford}, held at the Center of Mathematical Sciences and Applications, Harvard University, August 18-20, 2018.
\end{abstract}

\maketitle


\section{Introduction}

The theory of Bridgeland stability conditions (\cite{Bri07:stability_conditions}) has seen important developments in the past few years. Emerging from the mathematical physics literature (\cite{Dou02:mirror_symmetry}), it now connects to different branches in mathematics including symplectic geometry (\cite{BS15:quadratic_differentials,Joy15:conjectures_fukaya,HKK17:flat_surfaces,Smi17:stability_conjectures}) and representation theory (\cite{KS08:wall_crossing,ABM15:stability,Bri17:scattering}). 
This note gives a quick introduction to the basic theory through applications to problems in algebraic geometry.


A stability condition in the derived category is a direct generalization of the notion of stability for vector bundles on curves (\cite{Mum63:quotients}). The key property, Bridgeland's deformation theorem, is that stability conditions can be varied, and their variations form a complex manifold. More recent results (\cite{AP06:constant_t_structures,Tod08:K3Moduli,PT19:bridgeland_moduli_properties,AHLH18:moduli}) show that moduli spaces of semistable objects exist as proper algebraic spaces. Moreover, when the stability condition varies, the moduli space of semistable objects changes in a controlled way, giving rise to a locally-finite wall and chamber structure. Stability conditions and moduli spaces also exist in the relative setting (\cite{BLMNPS19:family}).
Finally, the existence of stability conditions is now known in interesting examples, including surfaces (\cite{Bri08:stability_k3,AB13:k_trivial}), certain Calabi--Yau threefolds (\cite{MP15:conjecture_abelian_threefoldsI, BMS16:abelian_threefolds,Li18:bg3_quintic}), Fano threefolds (\cite{Mac14:conjecture_p3, Sch14:conjecture_quadric, Li19:conjecture_fano_threefold, Piy17:Fano, BMSZ17:stability_fano}), some product varieties (\cite{Kos18:stability_products, Liu19:products}), and varieties with nef tangent bundles (\cite{Kos18:nef_tangent}).
We review the definition and basic properties of Bridgeland stability conditions in Section~\ref{sec:stability}.

The key approach to apply stability conditions to problems in algebraic geometry via wall-crossing was originally suggested in \cite{AB11:Reider}. The starting point is a certain limit point in the space of stability conditions where semistable objects are known, for example the large volume limit point, where stability essentially agrees with Gieseker stability for sheaves. Then the goal is to study how semistable objects vary when stability conditions move towards some other limit point determined by the problem in question. In favorable situations, this study is indeed possible, and leads to non-trivial results.

In this note, we show how to apply this approach to give new proofs for two fundamental results in algebraic geometry. The first application is to the Brill--Noether theorem by following \cite{Bay18:BrillNoether}. In this case, we look at stability in the derived category of a K3 surface. At the large volume limit point we aim to understand pure sheaves supported on curves with rank one. The theorem will follow once we understand the first wall where semistable objects change, and the geometry of its destabilized locus. This will be the subject of Section~\ref{sec:BN}.

The second application is to give bounds to the genus of space curves by following \cite{MS18:space_curves}.
Again, the starting point is the large volume limit point for certain weak stability conditions in the derived category of 
projective space, but we now look at ideal sheaves of curves.
If a curve has a too large genus, then its ideal sheaf must be destabilized at a certain point. Similarly as in the case of the Brill--Noether theorem, the aimed result follows once we are able to give a precise bound on when this could happen. More details will be given in Section~\ref{sec:genus}.

We do not claim any completeness in this short note.
There are many surveys available on the subject, starting from \cite{Bri06:icm} to the more specific lecture notes \cite{Huy11:intro_stability,Bay11:lectures_notes_stability,Bay16:short_proof,MS17:lectures_notes,MS18:ncK3}.


\subsection*{Acknowledgements}

The authors would like to thank Enrico Arbarello, Arend Bayer, and Paolo Stellari for many useful conversations.
This paper is an expanded version of the lecture of the first author at the conference \emph{From Algebraic Geometry to Vision and AI: A Symposium Celebrating the Mathematical Work of David Mumford}, held at the Center of Mathematical Sciences and Applications, Harvard University, August 18-20, 2018. The authors would like to thank the organizers Ching-Li Chai, David Gu, Amnon Neeman, Mark Nitzberg, Yang Wang, Shing-Tung Yau, Song-Chun Zhu, and to dedicate this paper to David Mumford.




\section{Stability Conditions}\label{sec:stability}

In this section, after briefly recalling the various notions of stability for sheaves in Section \ref{subsec:sheaves}, we define stability for objects in a triangulated category and review some of its basic properties (in Section \ref{subsec:Bridgeland}).

\subsection{Stability for sheaves}\label{subsec:sheaves}

In order to construct separated moduli spaces of vector bundles on curves, Mumford introduced the notion of \emph{stability} in \cite{Mum63:quotients}. The definition is surprisingly simple: to quote directly from \cite[Page 529]{Mum63:quotients}, a vector bundle $E$ on a smooth projective curve $C$ is \emph{stable} if all its subbundles are ``less ample'', i.e., for all proper subbundles $F\subset E$, we have
\begin{equation}\label{eq:MumfordStability}
\mu(F):=\frac{\deg(F)}{\rk(F)} < \mu(E).
\end{equation}

The set of all stable vector bundles of fixed rank and degree is naturally isomorphic to the set of points of a smooth quasi-projective variety. There is a Quot-scheme and an action of $\PGL_n$ such that stable vector bundles correspond precisely to stable points in the sense of Geometric Invariant Theory (see \cite{MFK94:git,New78:introduction_moduli}).

This notion of stability can also be reinterpreted in terms of unitary representations (over the complex numbers). The main result in \cite{NS65:stable_vector_bundles} states that a vector bundle over a curve is stable if and only if it comes from an irreducible projective unitary representation of the fundamental group of the curve.
Finally, by \cite{Don83:new_proof_Narasimhan_Seshadri} (based on and complementing \cite{AB82:Yang_Mills_curves}), stability can also be reinterpreted in terms of connections. A vector bundle $E$ over a curve is stable if and only if there is a unique unitary connection on $E$ having constant central curvature $-2\pi i \mu(E)$.

The theory of vector bundles on curves has applications in many areas of mathematics. For example, in algebraic geometry, there are applications to syzygies of curves and projective normality (see \cite{Mum70:quadratic_equations,Laz89:sampling_vector_bundles_techniques}).

\medskip

In higher dimension, there are several generalizations possible.
The first one (\cite{Tak72:Stability1}) is to directly extend the definition of stability for vector bundles on curves in \eqref{eq:MumfordStability} to torsion-free sheaves. In this case, the degree $\deg(E)$ is replaced by the second coefficient of the Hilbert polynomial of $E$, with respect to a fixed ample divisor $H$, or equivalently, in the case where the variety is sufficiently regular, by the pairing of the first Chern class of $E$ with the ample numerical class $H^{n-1} \cdot \c_1(E)$. This is called \emph{Mumford--Takemoto stability}, or shortly \emph{slope stability}. An important remark is that it does depend on the choice of the numerical class of $H$.

The fact that moduli spaces exist is much less clear, but a deep result by Donaldson \cite{Don85:anti_self_dual_Yang_Mills} and Uhlenbeck--Yau \cite{UY86:hermitian_Yang_Mills} states that a vector bundle $E$ on a smooth complex projective polarized variety $(X,H)$ is stable if and only if it admits an \emph{irreducible Hermitian--Einstein connection}. This is the so called \emph{Kobayashi--Hitchin correspondence}.

Gieseker, Maruyama, and Simpson (\cite{Gie77:vector_bundles,Mar77:stable_sheavesI,Sim94:moduli_representations}) introduced a definition of stability for coherent sheaves that directly generalizes the GIT approach, and therefore, produces well-behaved moduli spaces.
This stability, often referred to as \emph{Gieseker stability}, is based on the full Hilbert polynomial and does indeed correspond to stability with respect to an appropriate GIT problem (\cite{HL10:moduli_sheaves, LeP97:lectures_vector_bundles}).
Slope stable vector bundles are Gieseker stable and Gieseker stability also depends on the choice of the polarization. See \cite{Tha96:git_flips, DH98:vgit} for general results on variation of GIT and \cite{EG95:variation_surfaces, FQ95:variation_surfaces, MW97:thaddeus_principle} for complete results on how variation of stability changes the moduli spaces in the case of surfaces.

\subsection{Bridgeland stability}\label{subsec:Bridgeland}

We now consider the bounded derived category $\Db X:=\Db(\Coh X)$ of coherent sheaves on a smooth complex projective variety $X$ (\cite{GM03:homological_algebra,Huy06:fm_transforms,Ver67:verdier_thesis}).

The notion of stability for $\Db X$ comes from the mathematical physics literature. More precisely, as explained in \cite{Dou02:mirror_symmetry}, it is an attempt to understand Dirichlet branes in the context of Kontsevich's Homological Mirror Symmetry (\cite{Kon95:ICM}). The basic idea was that such branes of B-type correspond to objects in the derived category (\cite{Dou01:D-branes_N1susy}). The ones which are actually reached physically (BPS branes) are just a subset $\PP\subset\Db X$, and this subset must depend on so-called ``stringy K\"ahler data''.
These BPS branes can be understood at the large volume limit. They are roughly vector bundles with irreducible Hermitian--Einstein connections. By using the Kobayashi--Hitchin correspondence described in the previous section, they are exactly slope stable vector bundles, 

Since both the GIT approach and the differential geometry approach are more difficult to understand for derived categories, the starting point of Douglas' construction is to formally isolate the properties that stable objects in the derived category should satisfy and how they should change on continuous paths in the K\"ahler moduli space. He called this notion \emph{$\Pi$-stability}.
This definition was formulated mathematically by Bridgeland in \cite{Bri07:stability_conditions} and then further studied by Kontsevich--Soibelman in \cite{KS08:wall_crossing} in the context of counting invariants of Donaldson--Thomas type.

Let $K(\Db X)$ denote the Grothendieck group of $\Db X$, which is isomorphic to the Grothendieck group of $X$.
We abuse notation and denote the class of an object of $\Db X$ in $K(\Db X)$ with the same symbol. The next step is to fix a free abelian group of finite rank $\Lambda$ and a group homomorphism
\[
v\colon K(\Db X) \to \Lambda.
\]
A standard choice of $\Lambda$ is the \emph{numerical Grothendieck group} $K_{\mathrm{num}}(\Db X):=K(\Db X)/\ker(\chi)$ given as the quotient of $K(\Db X)$ by the kernel of the Euler pairing
\[
\chi(E,F):=\sum_i (-1)^i\,\dim\Ext^i(E,F).
\]

The following definition is the latest version from \cite{BLMNPS19:family}. Bridgeland's original definition in \cite{Bri07:stability_conditions} only contains conditions (i)-(iv). We recommend a first time reader to skip the technical parts (v)-(vii).

\begin{defn}\label{def:Bridgeland}
A \emph{Bridgeland stability condition} on $\Db X$ with respect to $(v,\Lambda)$ is a pair $\sigma=(Z,\PP)$ where
\begin{itemize}
    \item $Z\colon \Lambda \to \C$ is a group homomorphism, called \emph{central charge}, and
    \item $\PP=\cup_{\phi\in\R}\PP(\phi)$ is a collection of full additive subcategories $\PP(\phi)\subset \Db X$
\end{itemize}
satisfying the following conditions:
\begin{enumerate}
    \item for all nonzero $E\in\PP(\phi)$ we have $Z(v(E))\in \R_{>0}\cdot e^{i\pi\phi}$;
    \item for all $\phi\in\R$ we have $\PP(\phi+1)=\PP(\phi)[1]$;
    \item if $\phi_1>\phi_2$ and $E_j\in\PP(\phi_j)$, then $\Hom(E_1,E_2)=0$;
    \item (Harder--Narasimhan filtrations) for all nonzero $E\in\Db X$ there exists a finite sequence of morphisms
    \[
    0=E_0 \xrightarrow{s_1} E_1 \xrightarrow{s_2}\dots \xrightarrow{s_m} E_m=E
    \]
    such that the cone of $s_j$ is in $\PP(\phi_j)$ for some sequence $\phi_1>\phi_2>\dots>\phi_m$ of real numbers;
    \item\label{eq:supportproperty} (support property) there exists a quadratic form $Q$ on the vector space $\Lambda_{\R}$ such that
    \begin{itemize}
        \item the kernel of $Z$ is negative definite with respect to $Q$, and
        \item for all $E\in\PP(\phi)$ for any $\phi$ we have $Q(v(E))\geq0$;
    \end{itemize}
    \item\label{eq:openness} (openness of stability) for every scheme $T$ and every $E\in\mathrm{D}_{T\text{-perf}}(X\times T)$ the set
    \[
    \left\{t\in T \, :\, E_t\in \PP(\phi) \right\}
    \]
    is open;
    \item\label{eq:boundedness} (boundedness) for any $v\in\Lambda$ and $\phi\in\R$ such that $Z(v)\in \R_{>0}\cdot e^{i\pi\phi}$ the functor
    \[
    T \mapsto \mathfrak{M}_{\sigma}(v,\phi)(T):=\left\{E\in\mathrm{D}_{T\text{-perf}}(X\times T)\,:\, E_t\in\PP(\phi)\text{ and }v(E_t)=v, \text{ for all } t\in T \right\}
    \]
    is bounded.
\end{enumerate}
\end{defn}

\begin{rmk}
\label{rmk:def_stability}
(a) The objects in $\PP(\phi)$ are called \emph{$\sigma$-semistable} of phase $\phi$. The simple objects in the abelian category $\PP(\phi)$ are called \emph{$\sigma$-stable}.

(b) The phases of the first and last factor in the Harder--Narasimhan filtration of an object $E$ are denoted $\phi^+(E)$ and $\phi^-(E)$.

(c) The support property can be equivalently stated as follows: We fix a metric $\| - \|$ on $\Lambda_\R$. There exists a constant $C>0$ such that for all $E\in\PP$
\[
\| v(E)\| \leq C\cdot |Z(v(E)|.
\]

(d) Openness and boundedness imply that moduli spaces of (semi)stable objects exist, even if in general there is no GIT problem associated to such stability. By using work in \cite{Lie06:moduli,AP06:constant_t_structures,Tod08:K3Moduli}, it is a consequence of the general foundational theory developed in \cite{AHLH18:moduli} that the moduli spaces $M_\sigma(v,\phi)$ parametrizing S-equivalence classes of semistable objects of class $v$ and phase $\phi$ exist and are proper algebraic spaces. Moreover, if the morphism $v$ factors through $K_{\mathrm{num}}(\Db X)$, by the results in \cite{BM14:projectivity}, there is a real numerical Cartier divisor class $\ell_{\sigma}$ on  $M_\sigma(v,\phi)$ which is strictly nef (see \cite[Theorem 21.24 and Theorem 21.25]{BLMNPS19:family}).

(e) An interesting elementary result (\cite[Proposition 5.3]{Bri07:stability_conditions}), which is very useful in constructing examples of Bridgeland stability conditions is a reformulation of Definition \ref{def:Bridgeland} in terms of slope, thus extending the numerical definition of stability \eqref{eq:MumfordStability} formally to the derived category. More precisely, the extension-closed category $\AA:=\PP((0,1])$ generated by all semistable objects with phases in the interval $(0,1]$ is an abelian category. It is furthermore the heart of a bounded t-structure on $\Db X$. The real and imaginary parts of the central charge $Z$ behave like a degree and rank function on $\AA$: for a nonzero object $E\in\AA$, $\Im Z(E)\geq0$ and if $\Im Z(E)=0$, then $\Re Z(E) <0$. An object $E\in\AA$ is $\sigma$-semistable if and only if it is slope-semistable with respect to the slope $\mu_\sigma:=-\frac{\Re Z}{\Im Z}$. The converse is also true. Let $Z$ be a central charge on the heart of a bounded t-structure $\AA$ satisfying the above numerical properties. Then we define (semi)stable objects in $\AA$ as slope-(semi)stable and extend them by shifts to $\Db X$.
We obtain a stability condition in $\Db X$ once Harder--Narasimhan filtrations exist in $\AA$ and the remaining properties \ref{eq:supportproperty}, \ref{eq:openness}, \ref{eq:boundedness} are satisfied. When we want to stress the category $\AA$ in the definition of stability we use the notation $\sigma=(Z,\AA)$.

(f) The definition of stability conditions is more general and can be given for any triangulated category with certain regularity and base change properties, even in the relative context. See \cite{BLMNPS19:family} for more details.
\end{rmk}

Let $\Stab_{(\Lambda,v)}(\Db X)$ be the set of stability conditions on $\Db X$, with respect to a fixed $(\Lambda,v)$. We will use the simplified notation $\Stab(\Db X)$ when the dependence on $(\Lambda,v)$ is clear.
The set $\Stab(\Db X)$ is endowed with the coarsest topology such that $\phi^+(E)$, $\phi^-(E)$ for all $E \in \Db X$, and the map $\mathcal{Z}\colon\Stab(\Db X)\to \Hom(\Lambda,\C)$ given by $(Z,\PP)\mapsto Z$ are all continuous.
The main result of \cite{Bri07:stability_conditions} is then the following.

\begin{thm}[Bridgeland Deformation Theorem]\label{thm:BridgelandDefo}
The map $\mathcal{Z}\colon\Stab(\Db X)\to \Hom(\Lambda,\C)$ given by $(Z,\PP)\mapsto Z$ is a local homeomorphism. In particular, $\Stab(\Db X)$ is a complex manifold of dimension $\rk(\Lambda)$.
\end{thm}

\begin{rmk}\label{rem:BridgelandDefoThm}
(a) The theorem can be made more precise by explicitly describing the local structure in terms of the quadratic form $Q$. See \cite[Appendix A]{BMS16:abelian_threefolds} for more details.

(b) There are two continuous group actions on $\Stab(\Db X)$. The universal cover $\widetilde{\mathrm{GL}}_2^+(\R)$ acts from the right on $\Stab(\Db X)$ by extending the corresponding action of $\mathrm{GL}_2^+(\R)$ on $\Hom(\Lambda,\C)$. Since Definition~\ref{def:Bridgeland} behaves well with respect to autoequivalences, the group $\Aut(\Db X)$ acts by isometries from the left on $\Stab(\Db X)$. For details see \cite[Lemma 8.2]{Bri07:stability_conditions}.

(c) A fundamental property of stability conditions in the derived category is that, in contrast to the classical notions of stability for sheaves in higher dimension, there is a locally-finite wall and chamber structure in $\Stab(\Db X)$ (see \cite[Proposition 9.3]{Bri08:stability_k3}). More precisely, if we fix the numerical class $v$ and consider the moduli spaces $M_{\sigma}(v)$ as $\sigma$ varies in $\Stab(\Db X)$, then $M_{\sigma}(v)$ and $M_{\sigma'}(v)$ are isomorphic as long as $\sigma$ and $\sigma'$ are in the same chamber. We refer to Proposition~\ref{prop:structure} below for an explicit statement in the case of K3 surfaces of Picard rank one.

(d) In the original picture coming from Douglas' work there should exist a limit point for the space of stability conditions where stability reduces to the usual notions for sheaves. This is one of the starting points in the construction of stability conditions (\cite{Bri08:stability_k3,BMT14:stability_threefolds,BMS16:abelian_threefolds}). In the case of surfaces this is \cite[Proposition 14.2]{Bri08:stability_k3}, extended in \cite[Section 6.2]{Tod08:K3Moduli}.
See Proposition~\ref{prop:large_volume_limit} below in the case of K3 surfaces.
\end{rmk}

\begin{ex}
\label{ex:curves}
Let $C$ be a smooth projective curve.
Then $\sigma_0=(Z_0,\Coh C)$, where
\[
Z_0 (-) = - \deg(-) + i \rk(-),
\]
is a Bridgeland stability condition with respect to $\Lambda=N(\D^b(C))\cong \Z^2$ and
\[
v = (\rk,\deg)\colon K(\Db C)\to \Lambda.
\]
If the genus of $C$ is strictly positive, then, up to the action of $\widetilde{\mathrm{GL}}_2^+(\R)$, these are the only stability conditions (\cite{Bri07:stability_conditions,Mac07:curves}).
Semistable objects in $\Coh C$ are exactly torsion sheaves and slope semistable vector bundles.
\end{ex}

Stability conditions have been applied to problems in algebraic geometry, for instance to the study of quadratic differentials (\cite{BS15:quadratic_differentials}), to moduli spaces on the projective plane (\cite{ABCH13:hilbert_schemes_p2,CHW17:effective_cones_p2,LZ19:NewStabilityP2,Bou19:takahashi_conjecture}), and to the theory of Hyperk\"ahler manifolds (\cite{MYY14:stability_k_trivial_surfaces,YY14:stability_abelian_surfaces,BM14:projectivity,BM14:stability_k3,BLMNPS19:family}).
In this note, we will present two applications: how to give a new proof for the Brill--Noether theorem for curves (see Section \ref{sec:BN}) and how to bound the genus of space curves (see Section \ref{sec:genus}).


\section{The Brill--Noether Theorem}
\label{sec:BN}

The Brill--Noether Theorem is a fundamental result in the theory of curves. It does contain information on morphisms from a general curve to projective space of a given degree. The result was originally proved in \cite{GH80:BN} by using degeneration methods (a simpler proof is in \cite{EH83:BN}). See also \cite{CDPR12:BN} for a recent proof using tropical geometry techniques.
Instead, the approach to the Brill--Noether Theorem by Lazarsfeld in \cite{Laz86:brill_noether} is to use curves on K3 surfaces.
In this section, we present ideas of Bayer from \cite{Bay18:BrillNoether} to give a new proof of Lazarsfeld's theorem by using wall-crossing in Bridgeland stability. 
These techniques also lead to new results: for instance, there are applications to Mukai's program on reconstructing a K3 surface from a curve (\cite{ABS14:Mukai_program,Fey17:Mukai_program}) and to higher rank Clifford indices of curves (\cite{FL18:higher_Clifford}).


Let $X$ be a K3 surface. For simplicity of the exposition, we assume that $X$ has Picard rank one, i.e., $\Pic(X) = \Z \cdot H$ for some ample divisor $H$. Let $C$ be any smooth curve in the linear system $|H|$, $d \in \Z_{\geq 1}$, and $r \in \Z_{\geq 0}$. By definition the \emph{Brill--Noether variety} $W^r_d(C)$ is the closed subset of $\Pic_d(C)$ consisting of those degree $d$ lines bundle $L$ on $C$ for which $h^0(L) \geq r + 1$. The \emph{Brill--Noether number} is the naive expected dimension of $W^r_d(C)$ given by $\rho(r, d, g) := g - (r + 1)(g - d + r)$, where $g$ is the genus of $C$.

\begin{thm}[Lazarsfeld]
\label{thm:brill_noether}
The variety $W^r_d(C)$ is non-empty if and only if $\rho(r, d, g) \geq 0$. Moreover, in that case it has expected dimension $\min\{\rho(r, d, g),\, g\}$.
\end{thm}

The geometric structure of $W^r_d(C)$ is also completely known for a general curve $C$: the variety $W^r_d(C)$ is smooth away from $W^{r+1}_d(C)$, as well as integral when $\rho(r,d,g)\geq 1$.
We refer to the bibliographical notes of \cite[Section V]{ACGH85:curves} for an overview on the rich history of Brill--Noether theory, involving many authors including Kempf, Kleiman--Laksov, Fulton, Gieseker (and based on previous work by Severi, Castelnuovo, Petri, and on unpublished work by Mumford).


The wall-crossing techniques in \cite{Bay18:BrillNoether} deal more naturally with a subset of $W^r_d(C)$, and allow to treat singular curves as well.
Let $C\in|H|$ be any curve (integral, by assumption).
Let $V^r_d(C)$ denote the constructible set of pure sheaves $F \in \Coh X$ supported on $C$ with rank one, $h^0(F) = r + 1$, and $\chi(F) = d + 1 - g$. Moreover, we define
\[
V^r_d(|H|) = \bigcup_{C \in |H|} V^r_d(C).
\]

We will show how to prove the following theorem which is the key step in \cite{Bay18:BrillNoether}. 

\begin{thm}
\label{thm:Bayer}
Assume $0 < d \leq g - 1$. The set $V^r_d(|H|)$ is non-empty if and only if $\rho(r, d, g) \geq 0$. Moreover, in that case there is a morphism $V^r_d(|H|) \to M$, where $M$ is a non-empty open subset of a smooth projective irreducible holomorphic symplectic variety of dimension $2\rho(r, d, g)$. Finally, each fiber is isomorphic to a Grassmannian variety of $(r + 1)$-dimensional quotients of a vector space of dimension $g - d + 2r + 1$.
\end{thm}

A fair warning: we will simply describe the morphism, $V^r_d(|H|)$, and the fibers set-theoretically and ignore any further issues. To get Theorem~\ref{thm:brill_noether} from this requires arguments disjoint from stability. To stay within the scope of these notes we refer to \cite{Bay18:BrillNoether} for details and simply give some brief ideas on how Theorem \ref{thm:Bayer} implies Theorem \ref{thm:brill_noether}.

The first step is the reduction to $d \leq g - 1$ or equivalently $\chi(L) = d - g + 1 \leq 0$, in the case of a smooth curve $C$. If $\chi(L) > 0$, then by Serre duality $\chi(L^{\vee} \otimes \omega_C) = -\chi(L) < 0$. The degree of $L^{\vee} \otimes \omega_C$ is given by $2g - 2 - d$. Now $h^0(L) \geq r + 1$ if and only if $h^0(L^{\vee} \otimes \omega_C) = h^1(L) = h^0(L) - \chi(L) \geq r - d + g$. Note that $\rho(r, d, g) = \rho(r - d + g - 1, 2g - 2 - d, g)$, and we constructed a bijective correspondence $W^r_d(C) \to W^{r - d + g - 1}_{2g - 2 - d}(C)$. Note in particular, that if $\min\{\rho(r, d, g),\, g\} = g$, then there is no condition on the line bundle. In particular, all line bundles with the correct invariants are in $W^r_d(C)$ and this case is trivial.

The second step is to connect the statement about $V^r_d(|H|)$ to $V^r_d(C)$ for a single curve (here the curve $C$ could be even singular). The idea is to use the morphism $\varphi\colon V^r_d(|H|) \to |H| = \P^g$ that maps every such line bundle to its support. For any $C \in |H|$ the fiber $\varphi^{-1}(C)$ is precisely $V^r_d(C)$. Therefore, if $V^r_d(|H|)$ is empty, then so is $V^r_d(C)$. The converse is more complicated and requires techniques related to holomorphic symplectic varieties, analyzing the morphism $V^r_d(|H|)\to M$. The conclusion is that $\varphi$ is surjective and all its fibers have the same dimension. Given the description of $V^r_d(|H|)$ in Theorem~\ref{thm:Bayer}, we know
\begin{align*}
\dim V^r_d(|H|) &= \dim M + \dim \Gr(r + 1, g - d + 2r + 1) \\
&= 2\rho(r, d, g) + (r + 1)(g - d + r) \\
&= \rho(r, d, g) + g.
\end{align*}
This implies $\dim V^r_d(C) = \dim V^r_d(|H|) - g = \rho(r, d, g)$. Lastly, $V^r_d(C)$ is related to $W^r_d(C)$ by
\[
W^r_d(C) = \overline{V^r_d(C)} = \bigcup_{r' \geq r} V^{r'}_d(C).
\]

\begin{rmk}
A simple further argument, when $C$ is smooth, $d \leq g - 1$, and $r\geq1$, also implies that the locus in $W^r_d(C)$ consisting of those line bundles which are globally generated has also dimension $\rho(r, d, g)$, the original formulation of Theorem~\ref{thm:brill_noether}.
\end{rmk}


\subsection{Stability on K3 surfaces}

The strategy to prove Theorem~\ref{thm:Bayer} is to analyze stability of elements $F \in V^r_d(|H|)$ as torsion sheaves on $X$. Instead of classical notions of stability, we will use Bridgeland stability conditions on K3 surfaces. For details beyond this overview we refer to the original source \cite{Bri08:stability_k3}.

In order to construct abelian categories different from $\Coh X$ the theory of \emph{tilting} is used. We refer to \cite{HRS96:tilting} for technical details of tilting. For any $\beta \in \R$, we define
\begin{align*}
    \TT^{\beta} &:= \{E \in \Coh X : \ \text{all quotients $E \onto Q$ satisfy $\mu(Q) > \beta$}\}, \\
    \FF^{\beta} &:= \{E \in \Coh X : \ \text{all non-trivial subobjects $K \into E$ satisfy $\mu(K) \leq \beta$}\}. 
\end{align*}
The \emph{tilted category} $\Coh^{\beta}X$ is defined as the smallest extension closed subcategory of $\Db X$ containing both $\TT^{\beta}$ and $\FF^{\beta}$. Said differently, $\Coh^{\beta}X$ consists of all those complexes $E \in \Db X$ for which $\HH^i(E) = 0$ for $i \neq -1, 0$, $\HH^0(E) \in \TT^{\beta}$, and $\HH^{-1}(E) \in \FF^{\beta}$. This category turns out to be abelian, and a sequence of maps
\[
0 \to A \to E \to B \to 0
\]
is short exact if and only if the induced sequence of maps
\[
A \to E \to B \to A[1]
\]
is a distinguished triangle in $\Db X$.

Recall that the Mukai vector of a sheaf, or more generally an object $E \in \Db X$, is defined as
\[
v(E) = (v_0(E), v_1(E), v_2(E)) := \ch(E) \cdot \sqrt{\td(X)} = (r(E), c_1(E), \ch_2(E) + r(E)).
\]
For any classes $v, w \in K(\Db X)$ the Mukai pairing is given by
\[
\langle v, w \rangle := -\chi(v, w) = v_1 \cdot w_1 - v_0 \cdot w_2 - v_2 \cdot w_0.
\]
To define a stability condition we fix another real parameter $\alpha > 0$, and for any $E \in \Coh^{\beta}X$ set
\[
Z_{\alpha, \beta}(E) := \langle \exp(\beta H + i \alpha H), v(E) \rangle.
\]
The corresponding slope function is given by
\[
\nu_{\alpha, \beta}(E) := -\frac{\Re(Z_{\alpha, \beta}(E))}{\Im(Z_{\alpha, \beta}(E))} = \frac{v_2(E) - \beta H \cdot v_1(E) + \frac{\beta^2}{2} H^2 \cdot v_0(E) - \frac{\alpha^2}{2} H^2 \cdot v_0(E)}{H \cdot v_1(E) - \beta H^2 \cdot v_0(E)}.
\]

A class $\delta\in K(\Db X)$ is called \emph{spherical} if $\delta^2 = -2$.
The following result was first proved in \cite{Bri08:stability_k3} with respect to a slightly different definition of stability condition than the one given in this note.
The free abelian group $\Lambda$ is $K_{\mathrm{num}}(\Db X)$, while $v\colon K(\Db X)\to\Lambda$ is indeed the Mukai vector.
The additional properties on openness and boundedness follow from \cite{Tod08:K3Moduli}.
The support property can be found in \cite[Lemma 8.1]{Bri08:stability_k3}: it is there proved in the form stated in Remark~\ref{rmk:def_stability}(c).

\begin{thm}
\label{thm:stab_k3}
The pair $\sigma_{\alpha, \beta} = (\Coh^{\beta}X, Z_{\alpha, \beta})$ is a Bridgeland stability conditions if for all spherical classes $\delta$ with $\Im Z_{\alpha, \beta}(\delta) = 0$ and $\delta_0>0$, we have $\Re Z_{\alpha, \beta}(\delta) > 0$.
\end{thm}

It turns out that the hypothesis of Theorem~\ref{thm:stab_k3} on spherical classes is automatically fulfilled if $\alpha^2 H^2 > 2$. Moreover, stability conditions will also exist outside some ``holes'' in the $(\alpha,\beta)$-plane (these holes exactly correspond to points $(\alpha,\beta)$ where there is a spherical class $\delta$ with $Z_{\alpha, \beta}(\delta)=0$), but their construction is slightly more involved (we refer to \cite[Section 12]{Bri08:stability_k3}; see also \cite[Figure 1]{BB13:autoequivalences_k3} for a picture of the $(\alpha,\beta)$-plane).

The following result gives an intuition for the large volume limit point (see {\cite[Proposition 14.2]{Bri08:stability_k3}} and \cite[Section 6.2]{Tod08:K3Moduli}).

\begin{prop}
\label{prop:large_volume_limit}
Let $E \in \Db X$ have positive rank. Fix a real number $\beta < \mu(E)$. Then $E$ is in $\Coh X$ and Gieseker-(semi)stable if and only if $E$ is in $\Coh^{\beta}X$ and $\sigma_{\alpha, \beta}$-(semi)stable for sufficiently large $\alpha \gg 0$.
\end{prop}

The Mukai vector of a stable object satisfies serious restrictions.

\begin{lem}
\label{lem:necessary_k3}
Let $E$ be a $\sigma_{\alpha, \beta}$-stable or slope-stable object. Then $v(E)^2 \geq -2$.
\end{lem}

\begin{proof}
Since $E$ is stable and the canonical bundle $\omega_X$ is trivial, we have $\ext^2(E, E) = \hom(E, E) = 1$. Since $\Coh^{\beta}X$ is the heart of a bounded t-structure $\ext^i(E, E) = \ext^{2-i}(E, E) = 0 $ for $i < 0$ and $i > 2$. Therefore,
\[
v(E)^2 = -\chi(E, E) = \ext^1(E, E) - 2 \geq -2. \qedhere
\]
\end{proof}

The converse to Lemma~\ref{lem:necessary_k3} holds under a genericity condition on the stability condition. It is a much more difficult result and we will explain it further below.

By varying $\alpha$ and $\beta$ stability changes and non-trivial results can often be obtained from understanding this in detail. There is a locally finite wall and chamber structure such that stability does not change within each chamber. More precisely, for linearly independent classes $v, w$ their \emph{numerical wall} is defined as
\[
W(v, w) := \{ (\alpha, \beta) \in \R_{> 0} \times \R : \ \nu_{\alpha, \beta}(v) = \nu_{\alpha, \beta}(w) \}.
\]
Such a numerical wall is called an \emph{actual wall} for $v$ if the set of $\nu_{\alpha, \beta}$-semistable objects with class $v$ are not the same on both sides of $W$ and at $W$ itself. 
A technical remark, which we will slightly overlook in this note, is that walls might have ``holes'' and break, corresponding indeed to spherical classes being mapped to $0$ by $Z_{\alpha,\beta}$. Thus, actual walls might only consist of a subset of a numerical wall (in between two holes). This will happen, for example, in the proof of Theorem~\ref{thm:Bayer}. Lastly, to simplify notation we will write $W(E, F)$ instead of $W(v(E), v(F))$ for $E, F \in \Db X$.

By keeping this last remark in mind, the following result is then an easy computation.


\begin{prop}[Structure theorem for walls]
\label{prop:structure}
Let $v \in K(\Db X)$.
\begin{enumerate}
    \item Numerical walls are either semicircles with center on the $\beta$-axis or lines parallel to the $\alpha$-axis. If $v_0 \neq 0$, there is a unique vertical numerical wall at $\beta = \mu(v)$. If $v_0 = 0$ and $v_1 \neq 0$, there is no actual vertical wall.
    \item All numerical semicircular walls with respect to $v$ have their apex along the curve $\nu_{\alpha, \beta}(v) = 0$. This means the following:
    \begin{enumerate}
        \item If $v_0 = 0$ and $v_1 \neq 0$, then all walls are nested semicircles whose apex is along the ray $\beta = \frac{v_2}{v_1}$.
        \item If $v_0 \neq 0$ and $v^2 \geq 0$, then there are two sets of nested semicircles, one on each side of the vertical wall.
        \item If $v^2 < 0$, then all semicircular walls intersect $\nu_{\alpha, \beta}(v) = 0$ in both its apex and in the unique point $(\alpha, \beta)$ for which $Z_{\alpha, \beta}(v) = 0$.
    \end{enumerate}
\end{enumerate}
\end{prop}

\begin{figure}
    \centering
    \begin{minipage}{0.45\textwidth}
        \begin{tikzpicture}
            \begin{axis}[
                ticks = none,
                axis lines = center,
                xlabel = $\beta$,
                ylabel = $\alpha$,
                xmin = -3.6, xmax = 2.6,
                ymin = 0, ymax = 3,
                axis equal image,
                trig format plots = rad]
                \addplot[domain = 0:3, samples=100, dashed, color = black]({-1/2}, {x});
                \addplot[domain = 0:pi, samples=100, color = black]({5/2*cos(x) - 1/2}, {5/2*sin(x)});
                \addplot[domain = 0:pi, samples=100, color = black]({sqrt(17)/2*cos(x) - 1/2}, {sqrt(17)/2*sin(x)});
                \addplot[domain = 0:pi, samples=100, color = black]({3/2*cos(x) - 1/2}, {3/2*sin(x)});
                \addplot[domain = 0:pi, samples=100, color = black]({sqrt(5)/2*cos(x) - 1/2}, {sqrt(5)/2*sin(x)});
                \addplot[holdot] coordinates{(0, 1)({-5*sqrt(3)/4 - 1/2}, 5/4)({3*sqrt(2)/4 - 1/2}, {3*sqrt(2)/4})};
            \end{axis}
        \end{tikzpicture}
        \caption{Numerical walls for $v_0 = 0$ and $v_1 \neq 0$.}
    \end{minipage}\hfill
    \begin{minipage}{0.45\textwidth}
        \begin{tikzpicture}
            \begin{axis}[
                ticks = none,
                axis lines = center,
                xlabel = $\beta$,
                ylabel = $\alpha$,
                xmin = -9, xmax = 9,
                ymin = 0, ymax = 4.5,
                axis equal image,
                trig format plots = rad]
                \addplot[domain = 0:3, samples=100, dashed, color = black]({-sqrt(8)*cosh(x)}, {sqrt(8)*sinh(x)});
                \addplot[domain = 0:3, samples=100, dashed, color = black]({sqrt(8)*cosh(x)}, {sqrt(8)*sinh(x)});
                \addplot[domain = 0:pi, samples=100, color = black]({cos(x) - 3}, {sin(x)});
                \addplot[domain = 0:pi, samples=100, color = black]({sqrt(17)/2*cos(x) - 7/2}, {sqrt(17)/2*sin(x)});
                \addplot[domain = 0:pi, samples=100, color = black]({7/2*cos(x) - 9/2}, {7/2*sin(x)});
                \addplot[domain = 0:pi, samples=100, color = black]({cos(x) + 3}, {sin(x)});
                \addplot[domain = 0:pi, samples=100, color = black]({7/2*cos(x) + 9/2}, {7/2*sin(x)});
                \addplot[holdot] coordinates{({sqrt(3)/2 - 3}, 1/2)({3 - sqrt(3)/2}, 1/2)};
            \end{axis}
        \end{tikzpicture}
        \caption{Numerical walls for $v_0 \neq 0$ and $v^2 \geq 0$.}
    \end{minipage}\hfill
    \begin{minipage}{0.45\textwidth}
        \begin{tikzpicture}
            \begin{axis}[
                ticks = none,
                axis lines = center,
                xlabel = $\beta$,
                ylabel = $\alpha$,
                xmin = -4, xmax = 4,
                ymin = 0, ymax = 4,
                axis equal image,
                trig format plots = rad]
                \addplot[domain = -3:3, samples=100, dashed, color = black]({sqrt(2)*sinh(x)}, {sqrt(2)*cosh(x)});
                \addplot[domain = 0:pi, samples=100, color = black]({11/6*cos(x) - 7/6}, {11/6*sin(x)});
                \addplot[domain = 0:pi, samples=100, color = black]({11/6*cos(x) + 7/6}, {11/6*sin(x)});
                \addplot[domain = 0:pi, samples=100, color = black]({3/2*cos(x) - 1/2}, {3/2*sin(x)});
                \addplot[domain = 0:pi, samples=100, color = black]({sqrt(2)*cos(x)}, {sqrt(2)*sin(x)});
                \addplot[holdot] coordinates{(0, {sqrt(2)})};
            \end{axis}
        \end{tikzpicture}
        \caption{Numerical walls for $v^2 < 0$.}
    \end{minipage}
\end{figure}

The following result is the key technical ingredient in the proof; the statement as written is contained in \cite[Section 6]{BM14:projectivity} (the version for sheaves is in \cite{Yos01:stable_sheaves_abelian}). We omit the notation $\phi$ from the moduli space $M_{\sigma_{\alpha, \beta}}(v)$.

\begin{thm}[Mukai, O'Grady, Huybrechts, Yoshioka, Toda]
\label{thm:existence_k3}
Let $v$ be a primitive class, and let $\sigma_{\alpha, \beta}$ be a stability condition that does not lie on an actual wall for objects with Mukai vector $v$. Then $M_{\sigma_{\alpha, \beta}}(v)$, the moduli space of $\sigma_{\alpha, \beta}$-stable objects with class $v$, is a smooth projective irreducible holomorphic symplectic variety of dimension $v^2 + 2$. In particular, it is non-empty if and only if $v^2 \geq -2$.
\end{thm}

\subsection{The proof}

We define
\[
v = (0, H, d + 1 - g).
\]
Note that any $F \in V^r_d(|H|)$ satisfies $v(F) = v$.

\begin{lem}
\label{lem:classification_wall_L}
The largest wall for objects with Mukai vector $v$ is given by $W(\OO_X, v)$. Moreover, $F \in V^r_d(|H|)$ if and only if $F$ is a pure sheaf with $v(F) = v$ that is destabilized along $W(\OO_X, v)$ by a short exact sequence
\begin{equation}
\label{eq:destabilizing_sequence_L}
0 \to \OO_X^{\oplus (r + 1)} \to F \to G \to 0,
\end{equation}
where $G$ is stable along $W(\OO_X, v)$ with Mukai vector $(-r - 1, H, d - g - r)$.
\end{lem}

\begin{proof}
The wall $W(\OO_X, v)$ intersects the ray $\beta = 0$ for $\alpha^2 = \tfrac{1}{g - 1}$. Assume there is a wall for $\beta = 0$ and $\alpha^2 > \tfrac{1}{g - 1}$. However, $v_1 = H$ and by definition of $\Coh^0 X$ any destabilizing subobject $A \into F'$ for some $F'$ with $v(F') = v$ satisfies $v_1(A) \in \{0, H\}$. In either case, the subobject or quotient has infinite slope along this ray, while $F'$ has finite slope. This cannot define a wall. Together with Proposition~\ref{prop:large_volume_limit} this shows that any Gieseker-stable sheaf $F'$ with $v(F') = v$ (not just $F \in V^r_d(|H|)$) is semistable along $W(\OO_X, v)$. With this Mukai vector, Gieseker-stable simply means being a pure sheaf support on a curve $C \in |H|$.

Note that the point $\alpha^2 = \tfrac{1}{g - 1}$, $\beta = 0$ does not correspond to a stability condition due to $Z_{\alpha, \beta}(\OO_X) = 0$. This breaks the wall, and we will study the part with $\beta < 0$ and call it $W$. It is not hard to see, from Proposition~\ref{prop:structure}, that $\OO_X$ is stable along $W$.

Assume that $F \in V^r_d(|H|)$. Then $\hom(\OO_X, F) = r + 1$ and we get a morphism $\OO_X^{\oplus (r + 1)} \to F$. The first claim is that this map is injective in $\Coh^{\beta} X$ for $\beta < 0$ such that there is $\alpha > 0$ with $(\alpha, \beta) \in W$. Let $A \into K \into \OO_X^{\oplus (r + 1)}$ be a stable subobject of the kernel $K$ in $\Coh^{\beta} X$ with $\nu_{\alpha, \beta}(A) \geq \nu_{\alpha, \beta}(K)$ for such $(\alpha, \beta)$. Since $\OO_X^{\oplus (r + 1)}$ is semistable, we get $\nu_{\alpha,\beta}(A) \leq \nu_{\alpha, \beta}(\OO_X)$. However, we must have equality because otherwise the quotient of $K \into \OO_X^{\oplus (r + 1)}$ would make $F$ unstable along $W$. 
We can choose a quotient $\OO_X^{\oplus (r + 1)} \onto \OO_X$ such that $A \to \OO_X$ is non-trivial. Since $\OO_X$ is stable, this is a contradiction. Therefore, such an $A$ does not exist and in conclusion the kernel is trivial, i.e., the map is injective in $\Coh^{\beta} X$.

We define $G$ to be the quotient
\[
0 \to \OO_X^{\oplus (r + 1)} \to F \to G \to 0.
\]
We have to show that $G$ is stable along $W$. Being the quotient of two semistable objects with the same slope, it is certainly semistable. Assume $G$ has a stable subobject $A$ with the same slope along $W$. Then by definition of $\Coh^{\beta} X$ we get
\[
0 \leq \frac{v_1(A)H}{H^2} - \beta v_0(A) \leq 1 + \beta (r + 1)
\]
for $(\alpha, \beta) \in W$. Taking the limit $\beta \to 0$, we get $v_1(A) \in \{0, H\}$. By exchanging $A$ with the quotient $G/A$ if necessary, we can assume that $v_1(A) = 0$. By continuity and linearity of $Z_{\alpha, \beta}$ in $(\alpha, \beta)$ the complex numbers $Z_{\alpha, 0}(A)$ and $Z_{\alpha, 0}(G)$ have to be linearly dependent for $\alpha^2 = \tfrac{1}{g - 1}$. Since $Z_{\alpha, 0}(A)$ has infinite slope and $Z_{\alpha, 0}(G)$ has finite slope, this is only possible if $Z_{\alpha, 0}(A) = 0$ for $\alpha^2 = \tfrac{1}{g - 1}$. A straightforward computation shows $v_2(A) = v_0(A)$. Therefore, it is not too hard to see that $A = \OO_X^{\oplus v_0(A)}$.
We are done, if we can show that $\Hom(\OO_X, G) = \Hom(G, \OO_X) = 0$. The long exact sequence from applying $\Hom(\OO_X, \cdot )$ to (\ref{eq:destabilizing_sequence_L}) implies $\Hom(\OO_X, G) = 0$. If there was a non-trivial morphism $G \to \OO_X$, then there would be a non-trivial morphism $F \to \OO_X$. But that would imply that $F$ is unstable above $W$, a contradiction.

Assume vice-versa that $F'$ is a pure sheaf supported on a curve $C$ with $v(F') = v$ that is destabilized by a short exact sequence as in (\ref{eq:destabilizing_sequence_L}). The long exact sequence from applying $\Hom(\OO_X, \cdot )$ to (\ref{eq:destabilizing_sequence_L}) implies $h^0(F') = r + 1 + h^0(G) = r + 1$, since stability of $G$ along the wall implies $\hom(\OO_X, G) = 0$.
\end{proof}


\begin{proof}[Proof of Theorem~\ref{thm:Bayer}]
Assume $\rho(r, d, g) < 0$ and $V^r_d(|H|) \neq 0$. By Lemma~\ref{lem:classification_wall_L}, there is a stable object $G \in M_{\sigma}(-r - 1, H, d - g - r)$ for $\sigma$ along $W(\OO_X, v)$. By Lemma~\ref{lem:necessary_k3} this implies
\[
-2 \leq (-r - 1, H, d - g - r)^2 = 2\rho(r, d, g) - 2 < -2,
\]
a contradiction.

Assume $\rho(r, d, g) \geq 0$. We will show non-emptiness and describe the structure of $V^r_d(|H|)$ at the same time. Let $F \in V^r_d(|H|)$. Then by Lemma~\ref{lem:classification_wall_L} the sequence (\ref{eq:destabilizing_sequence_L}) is the Harder--Narasimhan filtration of $F$ below the wall. In particular, the object $G$ is uniquely determined for fixed $F$. A morphism $V^r_d(|H|) \to M^{\operatorname{stable}}_{\sigma}(-r - 1, H, d - g - r)$ for $\sigma$ along $W(\OO_X, v)$ is defined by $F \mapsto G$.

Let $\sigma'$ be a point above $W$ in a sufficiently small enough neighborhood. By Theorem~\ref{thm:existence_k3} the moduli space $M_{\sigma'}(-r - 1, H, d - g - r)$ is a smooth projective irreducible holomorphic symplectic variety of dimension $(-r - 1, H, d - g - r)^2 + 2 = 2\rho(r, d, g)$. Since stability is an open property, the locus of stable objects $M^{\operatorname{stable}}_{\sigma}(-r - 1, H, d - g - r)$ is an open subset of $M_{\sigma'}(-r - 1, H, d - g - r)$, hence is also smooth, irreducible, and has the same dimension.

Since both $\OO_X$ and $G$ are stable along $W$ with the same slope, we have $\Hom(G, \OO_X) = 0$ and $\Ext^2(G, \OO_X) = \Hom(\OO_X, G) = 0$. Thus, $\ext^1(G, \OO_X) = \langle G, \OO_X \rangle  = g - d + 2r + 1 > r + 1$. Let $\Ext^1(G, \OO_X)^{\vee} \onto V$ be an $r + 1$-dimensional quotient. Then there is a natural extension
\[
0 \to \OO_X \otimes V \to F' \to G \to 0.
\]
We claim that $F'$ is stable above $W$. Any destabilizing subobject $A$ of $F'$ above $W$ has to be semistable along $W$ with the same slope as $F'$. Since Jordan-H\"older filtrations have unique factors up to order, there are two possibilities. If $A$ does not contain $G$ as stable factor, then $A = \OO_X^{\oplus s}$. However, this simply does not destabilize $F'$ numerically above the wall. If $A$ does contain $G$ as a stable factor, then we have a quotient $F' \onto \OO_X$. Applying $\Hom(\cdot, \OO_X)$ to the defining sequence for $F'$ leads to the long exact sequence
\[
0 \to \Hom(F', \OO_X) \to \Hom(\OO_X, \OO_X \otimes V) \to \Ext^1(G, \OO_X)
\]
By construction the last map is injective with image $V^{\vee}$. Thus, $\Hom(F', \OO_X) = 0$.
\end{proof}

\begin{rmk}
The same approach works for arbitrary $v$ if we consider objects that are stable near the wall $W(\OO_X, v)$ (\cite{Fey17:Mukai_program}). To extend Brill--Noether statements to stable sheaves on curves of higher rank, one first needs to control their Harder--Narasimhan filtration near $W(\OO_X, v)$. The idea is then to prove a general bound on the dimension of the space of global sections purely in terms of the shape of the Harder--Narasimhan filtration near $W(\OO_X, v)$.
\end{rmk}


\section{The genus of space curves}
\label{sec:genus}

A classical subject in algebraic geometry is the attempt to classify space curves. As usual there are two steps to this. First, understand their discrete invariants. Second, study their moduli spaces, i.e., Hilbert schemes of curves.

Hilbert scheme of curves in $\P^3$ are notoriously badly behaved. For example, in \cite{Mum62:pathologies} Mumford constructs an open subset of an irreducible component parametrizing smooth curves that is non-reduced everywhere. How to handle problems such as these remains a big open question. Instead, we would like to concentrate on discrete invariants. This study goes all the way back to Noether, Halphen \cite{Hal82:genus_space_curves}, and Castelnuovo \cite{Cas37:inequality} (see \cite[IV, \S 6]{Har77:algebraic_geometry}).

\begin{thm}[Halphen, Gruson--Peskine]
Let $C \subset \P^3$ be a smooth curve of degree $d$ and genus $g$.
\begin{enumerate}
    \item If $C$ is contained in a plane, then $g = \tfrac{(d - 1)(d - 2)}{2}$.
    \item If $C$ is not contained in a plane, then $g \leq \tfrac{d^2}{4} - d + 1$.
    \item If $C$ is not contained in a plane or a quadric, then $g \leq \tfrac{d^2}{6} - \tfrac{d}{2} + 1$.
    \item There are smooth curves $C \subset \P^3$ of degree $d$ and genus $g$ whenever $0 \leq g \leq \tfrac{d^2}{6} - \tfrac{d}{2} + 1$.
\end{enumerate}
\end{thm}

This theorem was stated by Halphen (\cite{Hal82:genus_space_curves}), but the proof was incomplete; it was finally proved by Gruson and Peskine over one hundred years later in \cite{GP78:genre_courbesI, GP82:postulation_courbes_gauches}. In order to understand for which $d$ and $g$ there are smooth curves, it is left to understand curves on quadrics, but this is elementary. We refer to \cite{Har87:space_curvesII} for more details on all of this.

This theorem begs an immediate follow up question. Fix a positive integer $k$. What happens if $C$ is not contained in a surface of degree $l$ for any $l < k$. When $d > k(k-1)$ this was also solved by Gruson and Peskine in \cite{GP78:genre_courbesI} 
and Harris (\cite{Har80:joe_space_curves}). We gave a completely new proof using stability conditions in the derived category in \cite{MS18:space_curves}.

\begin{thm}[Gruson--Peskine, Harris]
\label{thm:GP}
Let $C \subset \P^3$ be a smooth curve of degree $d$ and genus $g$ that is not contained in a surface of degree $l < k$. If $d > k(k-1)$, then
\[
2d + g - 1 \leq \frac{d^2}{2k} + \frac{dk}{2} - \varepsilon(d, k),
\]
where
\[
\varepsilon(d, k) = \frac{f}{2}\left(k - f - 1 + \frac{f}{k}\right).
\]
\end{thm}

In \cite{MS18:space_curves} we prove a more general statement that includes a very similar bound for principally polarized abelian threefolds of Picard rank one (and holds for integral curves as well). Note how the left hand side of the inequality is given by $2d + g - 1 = \ch_3(\II_C)$. In fact, the analogous theorem on other threefolds bounds $\tfrac{\ch_3(\II_C)}{H^3}$, where $H$ is an ample divisor generating the Picard group.

What happens in the case $d\leq k(k-1)$ is still open. There is a precise conjectural bound (and examples of curves satisfying this bound) in the case $\frac 13 (k^2+4k+6)\leq d \leq k(k-1)$, due to Hartshorne and Hirschowitz (\cite{HH88:genus_bound}), but with the exception of a few cases (see e.g., \cite{GP82:postulation_courbes_gauches,Har88:stable_reflexive_3,Ell91:genre_maximal}), this is still open. In the case $d<\frac 13 (k^2+4k+6)$, a bound is known (an easy consequence of the Clifford Theorem), but it is not yet known if this bound is achieved in all cases (results towards this are in \cite{BBEMR97:maximum_genus_rangeA,BLS18:max_genus}). 

The original proofs of Theorem~\ref{thm:GP} are based on general position results. For instance, on the generalized trisecant lemma (\cite{Lau78:trisecant_lemma}) that roughly says: for each curve of degree $d$ that is not contained in a surface of degree $<k$, with $k(k-1) < d$, there is a hyperplane section which in the corresponding plane is not contained in a curve of degree $<k$.

Our approach is different and based on an idea of Mumford to prove the Kodaira vanishing theorem for surfaces by using the Bogomolov inequality (see the appendix of \cite{Rei78:bogomolov}). Stability conditions on $\Db \P^3$ are based as well on a Bogomolov-type inequality (see Theorem~\ref{thm:generalized_bogomolov}). The basic idea is to run the same argument as in the proof of the Kodaira vanishing theorem, which unfortunately in this case involves quite messy computations. Most of the difficulty though in \cite{MS18:space_curves} comes from the error term $\varepsilon(d, k)$. We will give a proof without this error term and the slightly weaker condition $d \geq k^2$. This makes the argument much simpler, but most of the techniques of the precise statement are already present.

\subsection{Tilt stability}

The definition of Bridgeland stability on K3 surfaces is easy to generalize to other surfaces as pointed out in \cite{AB13:k_trivial}. It is enough to simply replace the Mukai vector with the Chern character in all definitions. In higher dimensions this will not lead to a Bridgeland stability condition, but nonetheless a weaker notion of stability is still well-defined as pointed out in \cite{BMT14:stability_threefolds}. We will explain the differences.

To simplify notation, we define the twisted Chern character $\ch^{\beta}(E) := \ch(E) \cdot e^{-\beta H}$ for any $E \in \Db X$. Note that if $\beta \in \Z$ this is nothing but $\ch(E \otimes \OO_X(-\beta H))$. Furthermore, we identify $\ch_i(E)$ with $H^{3 - i} \cdot \ch_i(E)$ for any $E \in \Db \P^3$. In particular, multiplication is happening as numbers, not as cohomology classes.

The category $\Coh^{\beta}\P^3$ is defined in exactly the same manner as on K3 surfaces. The slope function is now given as
\[
\nu_{\alpha, \beta} := \frac{\ch_2^{\beta} - \frac{\alpha^2}{2} \ch_0^{\beta}}{\ch_1^{\beta}}.
\]
The sole reason that this does not lead to a Bridgeland stability condition is the fact that both denominator and numerator are zero in the case of sheaves supported in dimension zero.

In order to get a grip on the invariants of stable objects, we use the classical Bogomolov inequality (see \cite{Bog78:inequality}) in the following version ({\cite[Corollary 7.3.2]{BMT14:stability_threefolds}}):

\begin{thm}
\label{thm:bogomolov}
If $E \in \Coh^{\beta}\P^3$ is $\nu_{\alpha, \beta}$-semistable, then
\[
\Delta(E) := \ch_1(E)^2 - 2\ch_0(E)\ch_2(E) \geq 0. 
\]
\end{thm}

To get a handling of the third Chern character we require the following generalized Bogomolov inequality. A close version was proved in \cite{Mac14:conjecture_p3} and shown to be equivalent to the stated quadratic inequality in \cite{BMS16:abelian_threefolds}.

\begin{thm}
\label{thm:generalized_bogomolov}
If $E \in \Coh^{\beta}\P^3$ is $\nu_{\alpha, \beta}$-semistable, then
\[
Q_{\alpha, \beta}(E) := \alpha^2 \Delta_H(E) + 4 (\ch_2^{\beta}(E))^2 - 6 \ch_1^{\beta}(E) \ch_3^{\beta}(E) \geq 0. 
\]
\end{thm}

By design the equation $Q_{\alpha, \beta}(E) = 0$ is equivalent to $\nu_{\alpha, \beta}(E) = \nu_{\alpha, \beta}(\ch_1(E), 2\ch_2(E), 3\ch_3(E))$ and thus constitutes a numerical wall for $E$.

Walls behave very similarly to the case of K3 surfaces with a few key differences. Simply replace the Mukai vector with the Chern character and the Mukai form with $\Delta$ in the statement of Proposition~\ref{prop:structure}. Due to Theorem~\ref{thm:bogomolov} the quadratic form $\Delta$ is never negative for semistable objects, and therefore, the case of negative quadratic form in Proposition~\ref{prop:structure} can be ignored. Moreover, due to \cite[Corollary 3.11]{BMS16:abelian_threefolds} line bundles are stable for all $\alpha > 0$, $\beta \in \R$.



In order to study wall-crossing in tilt stability we will frequently use the following proposition to bound the rank of potentially destabilizing subobjects (\cite[Proposition 8.3]{CH16:ample_cone_plane} and \cite[Lemma 2.4]{MS18:space_curves}).

\begin{prop}
\label{prop:max_rank}
Let $E \in \Coh^{\beta}\P^3$ with $\ch_0(E) > 0$ be $\nu_{\alpha, \beta}$-semistable along some of its numerical walls $W$ with radius $\rho_W$. If $F$ is a $\nu_{\alpha, \beta}$-semistable subobject or quotient of $E$ with $\ch_0(F) > \ch_0(E)$, then
\[
\rho_W^2 \leq \frac{\Delta(E)}{4\ch_0(F)(\ch_0(F) - \ch_0(E))}.
\]
\end{prop}

\subsection{The proof}

Unfortunately, applying Theorem~\ref{thm:generalized_bogomolov} to $\II_C$ is not good enough to obtain a strong bound on the genus. However, it can give strong bounds for rank zero objects that will turn out useful.

\begin{prop}
\label{prop:rank_zero}
Let $E \in \Coh^{\beta}\P^3$ be $\nu_{\alpha, \beta}$-semistable for some $(\alpha, \beta)$ with $\ch(E) = (0, c, d, e)$. Then
\[
e \leq \frac{c^3}{24} + \frac{d^2}{2c}.
\]
\end{prop}

\begin{proof}
If $(\alpha, \beta)$ satisfy $\alpha^2 + (\beta - \tfrac{d}{c})^2 \leq \tfrac{c^2}{4}$, then $Q_{\alpha, \beta}(E) \geq 0$ implies the statement. Therefore, we may assume $\alpha^2 + (\beta - \tfrac{d}{c})^2 > \tfrac{c^2}{4}$. If we can show that there is no wall for $E$ with radius $> \tfrac{c}{2}$, then we could vary $\alpha, \beta$ to reduce to the previous case.

Assume $0 \to F \to E \to G \to 0$ is a short exact sequence inducing a wall $W$ with $\rho_W > \tfrac{c}{2}$. If $F$ is an object with $\ch_0(F) = 0$, then a straightforward computation shows that $\nu_{\alpha, \beta}(F) = \nu_{\alpha, \beta}(E)$ holds independently of $(\alpha, \beta)$. Therefore, such objects $F$ destabilize $E$ everywhere or nowhere and thus, cannot induce a wall. 

Assume that $\ch_0(F) > 0$. Then Proposition~\ref{prop:max_rank} implies
\[
\frac{c^2}{4} < \rho_W^2 \leq \frac{\Delta(E)}{4\ch_0(F)^2} = \frac{c^2}{4 \ch_0(F)^2} \leq \frac{c^2}{4},
\]
a contradiction. If $\ch_0(F) < 0$, the same calculation with $G$ instead of $F$ leads to a contradiction.
\end{proof}

The fact that $C$ is smooth allows to reduce the possible walls for $\II_C$ as follows.

\begin{lem}
\label{lem:destabilized_by_line_bundle}
Let $F \into \II_C$ be a rank one subobject in $\Coh^{\beta}\P^3$ destabilizing $\II_C$ along a semicircular actual wall $W(F, \II_C)$. Then $F$ is a line bundle.
\end{lem}

\begin{proof}
Let $G$ be the quotient of $F \into \II_C$ in $\Coh^{\beta}\P^3$. Taking the long exact sequence in cohomology, we get a long exact sequence of sheaves
\[
0 \to \HH^{-1}(F) \to 0 \to \HH^{-1}(G) \to \HH^0(F) \to \II_C \to \HH^0(G) \to 0.
\]
In particular, $\HH^{-1}(F) = 0$ and $F$ is a sheaf. By definition $\HH^{-1}(G) \in \FF^{\beta}$ is torsion-free, thus $F$ is also torsion-free. Since the quotient of $\HH^{-1}(G) \into F$ embeds into $\II_C$ it must be either trivial or $\HH^{-1}(G) = 0$. However, if it were trivial, then the map $F \to \II_C$ is trivial, in contradiction to being an injection. Overall, we showed that $0 \to F \to \II_C \to G \to 0$ is also a short exact sequence in $\Coh \P^3$.

There is a subscheme $W \subset \P^3$ of codimension at least two and a positive integer $m > 0$ such that $F = \II_W(-m) \into \II_C$. Since $C$ is integral, it is contained in either $W$ or a surface of degree $m$.

Assume that $C$ is contained in $W$. We can compute that $\Delta(F)$ is twice the degree of $W$. Similarly, $\Delta(\II_C)$ is twice the degree of $C$. However, this implies the contradiction $\Delta(F) \geq \Delta(\II_C)$. Indeed, the discriminant must decrease in this situation, for semistable subobjects of the same slope (see \cite[Corollary 3.10]{BMS16:abelian_threefolds}).

Assume that $C$ is contained in a degree $m$ surface. Then there is a morphism $\OO(-m) \into \II_C$. A straightforward computation shows that $W(\OO(-m), \II_C)$ is larger than or equal to $W(F, \II_C)$ and indeed, $\II_C$ is destabilized by a line bundle.
\end{proof}

It holds more generally that any destabilizing subobject of $\II_C$ for integral $C$ is reflexive (see \cite[Lemma 3.12]{MS18:space_curves}), but we do not require this statement here.

\begin{lem}
\label{lem:decreasing}
As long as $d \geq k^2$, the function
\[
E(d, k) = \frac{d^2}{2k} + \frac{dk}{2}
\]
is decreasing in $k$.
\end{lem}

\begin{proof}
This follows from the derivative of $E(d, k)$ by $k$.
\end{proof}

\begin{thm}
Let $C \subset \P^3$ be a smooth curve of degree $d$ and genus $g$ that is not contained in a surface of degree $l < k$. If $d \geq k^2$, then
\[
\ch_3(\II_C) \leq \frac{d^2}{2k} + \frac{dk}{2}.
\]
\end{thm}

\begin{proof}
Let $e = \ch_3(\II_C)$. Then $\ch(\II_C) = (1,0,-d,e)$. Assume for a contradiction
\[
e > \frac{d^2}{2k} + \frac{dk}{2}.
\]
We can compute $Q_{\alpha, \beta}(\II_C) \leq 0$ if and only if
\[
\alpha^2 + \left(\beta + \frac{3e}{2d} \right)^2 \leq \frac{9e^2 - 8d^3}{4d^2}.
\]
The equation $Q_{\alpha, \beta}(\II_C) = 0$ is the equation of a numerical wall, and therefore, another numerical wall is contained in $Q_{\alpha, \beta}(\II_C) < 0$ if its radius is smaller than the square root of the right hand side. By Theorem~\ref{thm:generalized_bogomolov} we get that any radius $\rho$ of a semicircular wall for $\II_C$ satisfies
\[
\rho^2 \geq \frac{9e^2 - 8d^3}{4d^2} > \frac{9(k^2 - d)^2}{16k^2} + \frac{d}{4} \geq \frac{d}{4} = \frac{\Delta(\II_C)}{8}.
\]
By Proposition~\ref{prop:max_rank} we get that $\II_C$ is destabilized by a subobject of rank one. By Lemma~\ref{lem:destabilized_by_line_bundle} it has to be a line bundle, too. Therefore, we can find an integer $h > 0$ such that $\II_C$ is destabilized via an exact sequence
\[
0 \to \OO(-h) \to \II_C \to Q \to 0.
\]
Applying Proposition~\ref{prop:rank_zero} to $Q$ leads to
\[
e \leq \frac{d^2}{2h} + \frac{dh}{2} = E(d, h).
\]
By Lemma~\ref{lem:decreasing} this leads to a contradiction unless $d < h^2$. The wall $W(\II_C, \OO(-h))$ is given by
\[
\alpha^2 + \left(\beta + \frac{h}{2} + \frac{d}{h} \right)^2 = \frac{(2d - h^2)^2}{4h^2}.
\]
For $d < \frac{h^2}{2}$, we get that the wall is to the right of $\beta = -h$. In particular, $\OO(-h)[1] \in \Coh^{\beta}\P^3$ and $\OO(-h)$ cannot be a subobject. Therefore, we are left to deal with the situation $\tfrac{h^2}{2} \leq d < h^2$. We can compute
\[
\frac{(2d - h^2)^2}{4h^2} < \frac{d}{4} < \frac{9e^2 - 8d^3}{4d^2}.
\]
Hence the wall lies in the region where $Q_{\alpha, \beta}(\II_C) < 0$, a contradiction to Theorem~\ref{thm:generalized_bogomolov}.
\end{proof}


\newcommand{\etalchar}[1]{$^{#1}$}
\def\cprime{$'$} \def\cprime{$'$}

\end{document}